\documentclass[a4paper,12pt]{amsart}
\usepackage{hyperref,fancyhdr,mathrsfs,amsmath,amscd,amsthm,amsfonts,latexsym,amssymb,stmaryrd}
\usepackage[all]{xy}

\DeclareMathOperator{\dif}{d}

\newcommand{\Cal}{\mathcal{C}}
\newcommand{\Dcal}{\mathcal{D}}
\renewcommand{\H}{\mathscr{H}}
\newcommand{\K}{\mathscr{K}}

\def \e{\eta}
\def \ep{\varepsilon}

\def \phi{\varphi}
\def \Phi{\varPhi}
\def \p{\pi}
\def \r{\rho}

\def \C{\mathbb{C}\,}

\def\widecheckg{g^{\hspace*{-2.5pt}\vbox to 5pt{\hbox to
0pt{\LARGE$\check{}$}}}\hspace*{2pt}}

\def\widecheckl{\lambda^{\hspace*{-3.5pt}\vbox to 8pt{\hbox to
0pt{\LARGE$\check{}$}}}\hspace*{2pt}}

\begin{document}

\title{A simple construction of generalized\\
complex manifolds}
\author{Radu~Pantilie}

\email{\href{mailto:Radu.Pantilie@imar.ro}{Radu.Pantilie@imar.ro}}
\address{R.~Pantilie, Institutul de Matematic\u a ``Simion~Stoilow'' al Academiei Rom\^ane,
C.P. 1-764, 014700, Bucure\c sti, Rom\^ania}

\newtheorem{thm}{Theorem}[section]
\newtheorem{lem}[thm]{Lemma}
\newtheorem{cor}[thm]{Corollary}
\newtheorem{prop}[thm]{Proposition}

\theoremstyle{definition}

\newtheorem{defn}[thm]{Definition}
\newtheorem{rem}[thm]{Remark}
\newtheorem{exm}[thm]{Example}

\numberwithin{equation}{section}

\maketitle
\thispagestyle{empty}
\vspace{-0.5cm}
\section*{Abstract}
\begin{quote}
{\footnotesize  We construct a natural generalized complex structure on the total space of any bundle
endowed with a Chern connection and whose typical fibre is a homogeneous symplectic manifold.
This extends constructions of \cite{AleDav-GC} and \cite{CavGua-nil} and leads to natural examples
of holomorphic maps between generalized complex manifolds.}
\end{quote}

\section*{Introduction} 

\indent 
The notion of Dirac structure is the natural generalization of the notions of symplectic and Poisson structures. This is achieved 
through the essentially equivalent notion of presymplectic foliation and, up to the integrability, it corresponds, on any 
manifold $M$, to a subbundle of $TM\oplus T^*M$ which is maximally isotropic with respect to the inner product induced by the 
natural identification of the tangent bundle with its bidual.\\ 
\indent 
On complexifying, one is led to maximal isotropic subbundles of \mbox{$T^{\C}\!M\oplus(T^{\C}\!M)^*$.} Obviously, any such 
subbundle $L$\,, which further satisfies $L\cap\overline{L}=0$\,, is the \mbox{${\rm i}$-eigenbundle} of an orthogonal complex structure on $TM\oplus T^*M$. 
Together with the suitable integrability condition, this gives the notion of \emph{generalized complex structure} (see~\cite{Gua-thesis}\,).\\ 
\indent 
Now, up to products and $B$-field transformations, the following seem to exhaust the known `God given' (that is, not manufactured) 
examples of such structures:\\ 
\indent 
\quad(i) the (classical) complex structures;\\  
\indent 
\quad(ii) the symplectic structures;\\  
\indent 
\quad(iii) the generalized complex structures associated to the holomorphic Poisson structures;\\ 
\indent 
\quad(iv) the generalized complex structures associated to the holomorphic foliations on K\"ahler manifolds.\\ 
\indent 
The aim of this note is to add one more natural construction to this list (Theorem \ref{thm:gc_C_hg}\,; see, also, Corollary \ref{cor:gc_C_hg}\,). 
This extends constructions of \cite{AleDav-GC} and \cite{CavGua-nil} 
and, also, leads to natural examples of holomorphic maps between generalized complex manifolds.\\ 
\indent 
Our construction works for any bundle endowed with a Chern connection and whose typical fibre is a homogeneous symplectic manifolds, 
and seems to, also, give examples of  `generalized holomorphic bundles' whose typical fibres are symplectic.

\section{The construction}

\indent
Let $(P,M,G^{\C})$ be a holomorphic principal bundle whose structural group is the complexification of a Lie group $G$.
Suppose that $(P,M,G^{\C})$ admits a (smooth) reduction $(Q,M,G)$ and that $G$ acts transitively, by symplectic diffeomorphisms, 
on a symplectic manifold $(F,\ep)$\,.\\ 
\indent
Let $E$ be the bundle with typical fibre $F$ associated to $Q$\,, through the action of $G$ on $F$. We shall, also, denote
by $\ep$ the induced symplectic structure on the foliation formed by the fibres of $E$.\\
\indent
Let $\H\subseteq TE$ be the connection on $E$ induced by the Chern connection (see \cite[p.\ 185]{KoNo-II}\,) of $Q$\,, determined by $P$. 
As $\p^{-1}(TM)=\H$, where $\p:E\to M$ is the projection, the complex structure of $M$ induces a decomposition $\H^{\C}=\H^{1,0}\oplus\H^{0,1}$.

\begin{thm} \label{thm:gc_C_hg}
$L\bigl(\H^{1,0}\oplus({\rm ker}\dif\!\p),{\rm i}\,\ep\bigr)$ is a generalized complex structure on $E$, where 
$\ep$ is extended, over $\H$, such that $\iota_X\ep=0$\,, for any $X\in\H$.
\end{thm}
\begin{proof} 
Let $\r:G\to F$ be the projection determined by the action of $G$, by fixing a point of $F$. Then $\e=\r^*(\ep)$ is 
left-invariant and therefore determines a presymplectic structure, on the foliation formed by the fibres of $Q$\,, 
which we shall, also, denote by $\e$.\\ 
\indent 
Let $\K\subseteq TQ$ be the Chern connection on $Q$\,, determined by $P$, and extend $\e$ over $\K$ such that 
$\iota_X\e=0$\,, for any $X\in\K$. Also, the complex structure of $P$ induces a decomposition 
$\K^{\C}=\K^{1,0}\oplus\K^{0,1}$.\\ 
\indent 
Denote $L_E=L\bigl(\H^{1,0}\oplus({\rm ker}\dif\!\p),{\rm i}\,\ep\bigr)$ and 
$L_Q=L\bigl(\K^{1,0}\oplus({\rm ker}\dif\!\p_Q),{\rm i}\,\e\bigr)$\,, where $\p_Q:Q\to M$ is the projection. 
By using \cite[Proposition 1.3]{ogc}\,, it is easy to see that $\r^*(L_E)=L_Q$ and $\r_*(L_Q)=L_E$\,, where $\r$\,, also, denotes 
the projection from $Q$ onto $E$. Consequently, $L_E$ is integrable if and only if $L_Q$ is integrable.\\ 
\indent 
Now, the integrability of $\K^{1,0}\oplus({\rm ker}\dif\!\p_Q)$ is an immediate consequence of the fact that the curvature 
form of a Chern connection is of type $(1,1)$\,.\\ 
\indent 
To prove that $\dif\!\e=0$\,, on $\K^{1,0}\oplus({\rm ker}\dif\!\p_Q)$\,, it is sufficient to consider two cases. 
For this, let $X,Y$ be basic sections of $\K^{1,0}$ and let $A,B$ be fundamental vector fields on $Q$\,. 
Then, as $[X,A]=[X,B]=0$ and $\e(A,B)$ is constant, we have $\dif\!\e(X,A,B)=X(\e(A,B))=0$\,. 
Also, by using, again, that the curvature form of $\K$ is of type $(1,1)$\,, we obtain $\dif\!\e(X,Y,A)=-\e([X,Y],A)=0$\,.\\ 
\indent 
The proof is complete. 
\end{proof} 

\indent 
We have two classes of concrete examples for our construction:\\ 
\indent 
\quad(I) $G=F$ is Abelian and endowed with an invariant symplectic structure. For example, if $G=\bigl(S^1\bigr)^{2k}$ 
then $P$ is the bundle of adapted frames of a holomorphic vector bundle of rank $2k$\,, over $M$, which splits 
into the direct sum of holomorphic line bundles. Consequently, $Q$ is the bundle of adapted frames of some Hermitian structure 
on the corresponding holomorphic vector bundle. Moreover, if the Hermitian structure is real-analytic then the fibres of $P$ 
endowed with $\ep^{\C\!}$ are the leaves of the symplectic foliation of the canonical (holomorphic) Poisson quotient 
(see \cite[Theorem 2.3]{ogc}\,) of the complexification of $(E,L_E)$\,.\\ 
\indent 
\quad(II) $F$ is a coadjoint orbit endowed with the symplectic structure induced by the canonical 
Poisson structure on the Lie algebra of $G$. For example, the total space of any flag bundle of the tangent bundle 
of a Hermitian manifold is endowed, by Theorem \ref{thm:gc_C_hg}\,, with a generalized complex structure.\\ 

\indent 
With the same notations as above let $\phi:N\to M$ be a holomorphic map. Then, on endowing $\phi^{-1}(E)$ with the 
generalized complex structure induced by $\phi^{-1}(P)$ and $\phi^{-1}(Q)$\,, the canonical bundle map 
$\phi^{-1}(E)\to E$ is holomorphic.\\ 
\indent 
Similarly, any equivariant Poisson morphism from $F$ to some homogeneous symplectic manifold determines a holomorphic map 
between generalized complex manifolds constructed as in Theorem \ref{thm:gc_C_hg}\,.\\ 
\indent 
We end by noting that case (I)\,, above, can be strengthen, as follows. 

\begin{cor} \label{cor:gc_C_hg}
Let $G=\bigl(S^1\bigr)^{2k}$ and let $(Q,M,G)$ be a principal bundle. 
Then there exists a natural bijective correspondence between the following:\\ 
\indent 
{\rm (i)} Pairs formed of an invariant symplectic structure on $G$ and a holomorphic structure on the extension of $Q$ to $G^{\C}\!$;\\ 
\indent 
{\rm (ii)} $G$-invariant generalized complex structures, in normal form, on $Q$ 
for which the foliations of the associated Poisson structures are given by the fibres of~$Q$\,. 
\end{cor}  
\begin{proof} 
It is sufficient to prove that any generalized complex structure $L$ on $Q$ satisfying (ii) is obtained as in 
Theorem \ref{thm:gc_C_hg}\,.\\ 
\indent 
Firstly, as $L$ is $G$-invariant, its Poisson structure is, also, $G$-invariant. Together with the fact that $L$ is integrable 
this implies that the Poisson structure of $L$ is induced by some invariant symplectic structure on $G$.\\ 
\indent 
Let $\Dcal$ be the co-CR structure determined by $L$. Then $\dif\!\p(\Dcal)$ is a complex structure on $M$, where 
$\p:Q\to M$ is the projection.\\ 
\indent 
Also, if $\Cal$ is the CR structure determined by $L$ then $\K^{\C}=\Cal\oplus\overline{\Cal}$, where $\K$ is a principal connection 
on $Q$\,. Moreover, the fact that $\Cal$ is integrable is equivalent to the fact that the curvature form of $\K$ is of type $(1,1)$\,.\\ 
\indent 
Thus, by a classical result, due to Koszul and Malgrange, there exists a unique holomorphic structure on the extension 
of $Q$ to $G^{\C}$ whose Chern connection is~$\K$. 
\end{proof}

\end{document}